\documentclass{amsart}
\usepackage{amssymb,amscd,graphics,axodraw}
\usepackage[enableskew]{youngtab}
\usepackage[usenames]{color}
\usepackage[all]{xy}

\numberwithin{equation}{section}

\newtheorem{theorem}{Theorem}
\newtheorem{proposition}[theorem]{Proposition}
\newtheorem{lemma}[theorem]{Lemma}

\theoremstyle{definition}

\newtheorem{remark}{Remark}
\newtheorem{example}[remark]{Example}

\numberwithin{theorem}{section}
\numberwithin{definition}{section}
\numberwithin{remark}{section}

\newcommand{\calB}{\mathcal{B}}
\newcommand{\id}{\mathrm{id}}
\newcommand{\la}{\lambda}
\newcommand{\lb}{\mathrm{lb}}
\newcommand{\lh}{\mathrm{lh}}
\newcommand{\ls}{\mathrm{ls}}

\newcommand{\ot}{\otimes}

\newcommand{\Path}{\mathcal{P}}

\newcommand{\Q}{\mathbb{Q}}

\newcommand{\R}{\mathbb{R}}
\newcommand{\RC}{\mathcal{RC}}
\newcommand{\tr}{\mathrm{tr}}
\newcommand{\veps}{\varepsilon}
\newcommand{\vphi}{\varphi}
\newcommand{\Z}{\mathbb{Z}}

\newcommand{\diagram}[8]{\xymatrix{
	#1 \ar[r]^{#5} \ar[d]_{#6} & #2 \ar[d]^{#7} \\
	#3 \ar[r]_{#8} & #4}}

\begin{document}

\title[Similarity and KSS bijection]{Similarity and Kirillov-Schilling-Shimozono bijection}

\author[M.~Okado]{Masato Okado}
\address{Department of Mathematics, Osaka City University, 
3-3-138, Sugimoto, Sumiyoshi-ku, Osaka, 558-8585, Japan}
\email{okado@sci.osaka-cu.ac.jp}

%\date{\today}
%\footnote[0]{2010 {\it Mathematics Subject Classification.} 
%Primary 17B37 82B23 05A19; Secondary 17B25 81R50 81R10 05E10 11B65.}

\begin{abstract}
The behavior of the Kirillov-Schilling-Shimozono bijection is examined under 
the similarity map on Kirillov-Reshetikhin crystals. It enables us to define this 
bijection over $\Q$. Conjectures on the extension to $\R$ is also presented.
\end{abstract}

\maketitle

%\tableofcontents

\section{Introduction}

The Kerov-Kirillov-Reshetikhin bijection \cite{KKR,KR}, or Kirillov-Schilling-Shimozono
bijection in more general setting \cite{KSS}, describes a one-to-one correspondence
between highest weight elements of a multiple tensor product of Kirillov-Reshetikhin
crystals of type A \cite{KMN2:1992,Sh:2002} and certain combinatorial objects called 
rigged configurations. Let $B$ be a tensor product of 
Kirillov-Reshetikhin crystals and $\Path(B)$ be the set of highest weight elements
of $B$. Then there corresponds a datum $L(B)$, the set $\RC(L(B))$ of rigged
configurations associated to $L(B)$, and
there exists a bijection
\[
\Phi:\Path(B)\longrightarrow\RC(L(B)).
\]

The Kirillov-Schilling-Shimozono bijection, KSS bijection for short, has various 
applications. There exist notions of weight and energy or charge statistic on both
sets, and $\Phi$ preserves them. Hence, taking generating functions with a fixed
weight give rise to an identity known as $X=M$ \cite{HKOTY,HKOTT}. Using the
Kyoto path model \cite{KMN1:1992} $M$ in a suitable limit gives an explicit form
of a branching function of the highest weight $\widehat{\mathfrak{sl}}_n$-module
with respect to the underlying simple Lie algebra $\mathfrak{sl}_n$. Another
significant application is found in the box-ball system \cite{TS}, where the bijection
$\Phi$ linearize this nonlinear ultra-discrete dynamical system \cite{KOSTY}.

In \cite{Oka} we reported that Kirillov-Reshetikhin (KR) crystals have the similarity
property. Let $B^{r,s}$ stand for a KR crystal where $r$ is an index of the Dynkin
diagram and $s$ a positive integer. The similarity map $S_m$ 
($m\in\Z_{>0}$) is an injective map $S_m:B^{r,s}\rightarrow B^{r,ms}$ satisfying
some properties on crystal operators. For type A, representing elements of $B^{r,s}$ by 
semistandard tableaux of $r\times s$ rectangular shape, the image of $S_m$ is
obtained by enlarging it horizontally $m$ times. 
Then it is a natural question to ask how
the composition map $\Phi\circ S_m$ is described. The answer is quite simple
and natural. A rigged configuration $(\nu,J)$ is composed of a sequence of partitions
$\nu$ and a set of nonnegative integers $J$. The similarity map $S_m$ on $\RC(L(B))$
amounts to multiplying by $m$ each part of all partitions in $\nu$ and each integer in 
$J$ (Theorem \ref{th:main}).

The above theorem enables us to consider the KSS bijection over $\Q$. Namely,
we represent an element of $B^{r,s}$ and of $\Path(B)$ as an integer point of a 
certain polytope in the Euclidian space. We then consider any rational points of the
polytope, apply $\Phi\circ S_m$ ($m$ is chosen so that the image of $S_m$ is an 
integer point), and shrink by $m$ on the rigged configuration side. We can show 
the map so constructed is well-defined and extends the bijection from integer points
to rational ones (Proposition \ref{prop:Q}). Furthermore, one should be able to consider 
$\Phi$ over $\R$. However, the proof of well-definedness seems nontrivial and 
we leave it to future problem. We end with conjectures on this extension to $\R$ 
and a connection to the tropical periodic Toda lattice by \cite{T}.

\section{Reviews on KR crystals, paths and rigged configurations}

\subsection{KR crystal} \label{subsec:KRcrystal}

A Kirillov-Reshetikhin crystal, KR crystal for short, is a crystal basis of certain 
finite-dimensional quantum affine algebra 
module called Kirillov-Reshetikhin module. If the corresponding affine algebra is of
nonexceptional type, its existence was shown in \cite{OS:2008}. A KR crystal is 
denoted by $B^{r,s}$, since it is parametrized by $(r,s)\in I\setminus\{0\}\times\Z_{>0}$
where $I$ is the index set of the Dynkin diagram of the affine algebra and $0$ is the
index as prescribed in \cite{Kac}. In this paper we denote Kashiwara operators by
$e_i,f_i$.

In \cite{Oka} we showed the following property of KR crystals of nonexceptional type.

\begin{theorem} \label{th:similarity}
For $m\in\Z_{>0}$ there exists a unique injective map 
\[
S_m:B^{r,s}\longrightarrow B^{r,ms}
\]
satisfying 
\[
S_m(e_ib)=e_i^mS_m(b), S_m(f_ib)=f_i^mS_m(b)
\]
for $i\in I$ and $b\in B^{r,s}$. Here $S_m(0)$ is understood to be $0$.
\end{theorem}

In what follows we consider the case of $A_n^{(1)}$. Set $I_0=I\setminus\{0\}$.
Let $\alpha_i\,(i\in I)$ be simple roots
and $\varpi_i\,(i\in I_0)$ level 0 fundamental weights. We describe the KR crystal 
$B^{r,s}$ of $A_n^{(1)}$. As a set $B^{r,s}$ is given by semistandard
tableaux of shape $(s^r)$ with letters from $\{1,2,\ldots,n+1\}$. The action of 
Kashiwara operators $e_i,f_i$ for $i\in I_0$ is described in \cite{KN:1994} by reading letters
of the tableau in a certain manner. For $e_0,f_0$ it is defined through the so called 
promotion operator $\mathrm{pr}$ \cite{Sh:2002} by
\[
e_0=\mathrm{pr}^{-1}\circ e_1\circ \mathrm{pr},\quad
f_0=\mathrm{pr}^{-1}\circ f_1\circ \mathrm{pr}.
\]
See also \cite[\S2.2]{O:Memoirs} on these descriptions.
For an element $b$ of $B^{r,s}$ $S_m(b)$ is described as follows.
For each row a node with letter $a$ is replaced with $m$ nodes with the same 
letter $a$.

\subsection{Path} \label{subsec:path}

Let $B_1,B_2$ be crystals. The tensor product of crystals $B_2\ot B_1$ is defined
with its crystal structure given by
\begin{align}
e_i(b_2\ot b_1)&=\left\{
\begin{array}{ll}
b_2\ot e_ib_1\quad & \text{if }\veps_i(b_2)\le\vphi_i(b_1) \\
e_ib_2\ot b_1\quad & \text{if }\veps_i(b_2)>\vphi_i(b_1),
\end{array}\right. \label{e tensor}\\
f_i(b_2\ot b_1)&=\left\{
\begin{array}{ll}
b_2\ot f_ib_1\quad & \text{if }\veps_i(b_2)<\vphi_i(b_1) \\
f_ib_2\ot b_1\quad & \text{if }\veps_i(b_2)\ge\vphi_i(b_1).
\end{array}\right. \label{f tensor}
\end{align}
Here $0\ot b$ and $b\ot 0$ are understood by $0$ and
\[
\veps_i(b)=\max\{k\ge0\mid e^mb\ne0\},\quad
\vphi_i(b)=\max\{k\ge0\mid f^mb\ne0\}.
\]
From \eqref{e tensor}, \eqref{f tensor} we have
\begin{align}
\veps_i(b_2\ot b_1)&=\max(\veps_i(b_1),\veps_i(b_1)+\veps_i(b_2)-\vphi_i(b_1)),
\label{eps tensor}\\
\vphi_i(b_2\ot b_1)&=\max(\vphi_i(b_2),\vphi_i(b_1)+\vphi_i(b_2)-\veps_i(b_2)).
\label{phi tensor}
\end{align}
As seen from above we use the anti-Kashiwara convention for the tensor product
of crystals.

It is known \cite{KMN1:1992} that for KR crystals $B^{r,s},B^{r',s'}$ there exists an isomorphism of crystals
\[
R:B^{r,s}\ot B^{r',s'}\longrightarrow B^{r',s'}\ot B^{r,s},
\]
called combinatorial $R$-matrix. $R$ commutes with $e_i,f_i\,(i\in I)$. The image
of $R$ is given as follows \cite{Sh:2002}. Suppose $R(b\ot b')=\tilde{b}'\ot\tilde{b}$.
Then $b\rightarrow row(b')=\tilde{b}'\rightarrow row(\tilde{b})$, where $row(b)$
is the row word of $b$ and $T\rightarrow wd$ stands for the tableau given by row
inserting the word $wd$ into $T$. This condition uniquely determines $\tilde{b}'$ and $\tilde{b}$
from $b\ot b'$.

Let $B=B_k\ot\cdots\ot B_1$ be a multiple tensor product of KR crystals.
An element of $b$ of $B$ is called a (highest-weight) path if $e_ib=0$ for any
$i\in I_0$. The set of paths in $B$ is denoted by $\Path(B)$.

\begin{example} \label{ex:path}
\[
b=\Yvcentermath1
\young(112,234)\ot\young(2,3)\ot\young(1113)
\ot\young(2)\ot\young(1)
\smallskip
\]
is an element of $\Path(B^{2,3}\ot B^{2,1}\ot B^{1,4}\ot(B^{1,1})^2)$ of weight
$6\varpi_1+4\varpi_2-(4\alpha_1+4\alpha_2+\alpha_3)$.
\end{example}

\subsection{Rigged configuration}

We concentrate on rigged configurations of type $A_n^{(1)}$.  Let $(C_{ab})_{a,b\in I_0}$ be the Cartan
matrix of $A_n$, that is, $C_{ab}=2\delta_{a,b}-\delta_{a,b+1}-\delta_{a,b-1}$. Consider a matrix
$L=(L_i^{(a)})_{a\in I_0,i\in\Z_{>0}}$ of nonnegative integers, almost all zero. $L$ is called a multiplicity array.
Let $\nu=(m_i^{(a)})$ be another such matrix. Say that $\nu$ is an admissible configuration if
it satisfies
\begin{equation} \label{ppos}
  p_i^{(a)} \ge 0\qquad\text{for any $a\in I_0$ and
  $i\in\Z_{>0}$,}
\end{equation}
where
\begin{equation} \label{p}
p_i^{(a)} = \sum_{j\in\Z_{>0}} \left( L_j^{(a)} \min(i,j) -
\sum_{b\in I_0}C_{ab}\min(i,j)m_j^{(b)}\right).
\end{equation}
$p^{(a)}_i$ is called a vacancy number.

Let $\nu=(m^{(a)}_i)_{a\in I_0,i\in\Z_{>0}}$ be an admissible configuration. We identify $\nu$
with a sequence of partitions $(\nu^{(a)})_{a\in I_0}$ such that
\begin{equation} \label{nu}
\nu^{(a)}=(1^{m_1^{(a)}}2^{m_2^{(a)}}\cdots).
\end{equation}
One can also identify the partition $\nu^{(a)}$ with a Young diagram whose number of rows of
length $i$ is $m^{(a)}_i$. A rigging $J$ on $\nu$ is to associate, with each part of the 
Young diagram $\nu^{(a)}$ of the same width $i$, a partition $(J^{(a,i)}_1\ge J^{(a,i)}_2\ge\ldots\ge J^{(a,i)}_{m^{(a)}_i})$
of length at most $m^{(a)}_i$ such that $p^{(a)}_i\ge J^{(a,i)}_1$. A pair $(\nu,J)$ of an admissible configuration $\nu$ and
a rigging $J$ on $\nu$ is called a rigged configuration.

For a partition $\mu$ and $i\in\Z_{>0}$, define
\begin{equation} \label{Qdef}
Q_i(\mu)=\sum_j \min(\mu_j,i),
\end{equation}
the area of $\mu$ in the first $i$ columns. Then the vacancy number \eqref{p} is rewritten as
\begin{equation} \label{p and Q}
p^{(a)}_i=Q_i(L^{(a)})+Q_i(\nu^{(a-1)})+Q_i(\nu^{(a+1)})-2Q_i(\nu^{(a)}).
\end{equation}
Here $L^{(a)}$ is a partition $(1^{L_1^{(a)}}2^{L_2^{(a)}}\cdots)$, and both $\nu^{(0)}$ and $\nu^{(n+1)}$
should be considered as an empty partition. The set of rigged configurations with multiplicity array $L$ 
is denoted by $\RC(L)$. We define a weight of the rigged configuration $(\nu,J)$ by
\[
\la = \sum_{a\in I_0}\left(Q_\infty(L^{(a)})\varpi_a-Q_\infty(\nu^{(a)})\alpha_a\right).
\]
It does not depend on the rigging $J$.

\begin{example} \label{ex:RC}
The following diagrams show an example of rigged configuration of
weight $6\varpi_1+4\varpi_2-(4\alpha_1+4\alpha_2+\alpha_3)$.
\vspace{7mm}
\begin{center}
\unitlength 10pt
\begin{picture}(20,5)
\Yboxdim{7pt}
\put(1,4){\yng(4,1,1)}
\put(7,4.7){\yng(3,1)}
\Yboxdim{10pt}
\put(1,1){\yng(3,1)}
\put(0.2,1){1}
\put(0.2,2){1}
\put(2.2,1){0}
\put(4.2,2){0}
\put(7,1){\yng(3,1)}
\put(6.2,1){1}
\put(6.2,2){1}
\put(8.2,1){1}
\put(10.2,2){0}
\put(13,2){\yng(1)}
\put(12.2,2){0}
\put(14.2,2){0}
\end{picture}
\end{center}
Upper diagrams are $L^{(1)},L^{(2)},L^{(3)}$ ($L^{(3)}$ is empty) drawn from left 
to right and lower ones are $\nu^{(1)},\nu^{(2)},\nu^{(3)}$. On the left (resp. right) of each
row the corresponding vacancy number (resp. rigging) is written.
\end{example}

\section{KSS bijection and similarity}

\subsection{Operations on paths}

We define several operations on paths. We first define
$\lh,\lb^{(s)},\ls^{(m)}$. In this subsection $B$ is a tensor product of KR crystals.

\begin{enumerate}
\item Suppose $B=B^{1,1}\ot B',b=c\ot b'\in B^{1,1}\ot B'$. The map 
$\lh:B\rightarrow B'$ is defined by $\lh(b)=b'$. We set $\lh(B)=B'$.
\item Suppose $B=B^{r,s}\ot B',b=c\ot b'\in B^{r,s}\ot B'\,(r\ge2)$. The map 
$\lb^{(s)}:B\rightarrow B^{1,s}\ot B^{r-1,s}\ot B'$ is defined by 
$\lb^{(s)}(b)=c'\ot c''\ot b'$, where $c'$ is the lowest row of $c$ and $c''$ is
obtained by removing $c'$ from $c$. We set $\lb^{(s)}(B)=B^{1,s}\ot B^{r-1,s}\ot B'$.
\item Suppose $B=B^{r,s}\ot B',b=c\ot b'\in B^{r,s}\ot B'\,(s\ge2)$. The map 
$\ls^{(m)}:B\rightarrow B^{r,m}\ot B^{r,s-m}\ot B'\,(1\le m<s)$ is defined by 
$\ls^{(m)}(b)=c'\ot c''\ot b'$, where $c'$ is the leftmost $m$ columns of $c$ and 
$c''$ is obtained by removing $c'$ from $c$. We set $\ls^{(s)}(B)=
B^{r,m}\ot B^{r,s-m}\ot B'$.
\end{enumerate}
These maps send a path to another path. In \cite{Sch:Memoirs} operations 
$\lb^{(s)},\ls^{(m)}$ were defined only when $s=1$ and $m=1$. 
We need the extensions of them to prove our main result. 
We set $\lb=\lb^{(1)},\ls=\ls^{(1)}$. 

\begin{example}
For a path $b$ in Example \ref{ex:path} we have $b'=\ls(b),b''=\lb(b'),
b'''=\lh(b'')$ as follows.
\begin{align*}
b'&=\Yvcentermath1
\young(1,2)\ot\young(12,34)\ot\young(2,3)\ot\young(1113)
\ot\young(2)\ot\young(1) \\
\smallskip
b''&=\Yvcentermath1
\young(2)\ot\young(1)\ot\young(12,34)\ot\young(2,3)\ot\young(1113)
\ot\young(2)\ot\young(1) \\
\smallskip
b'''&=\Yvcentermath1
\young(1)\ot\young(12,34)\ot\young(2,3)\ot\young(1113)
\ot\young(2)\ot\young(1)
\smallskip
\end{align*}
\end{example}

For later use we need the transpose of a path $b$ given in \cite{Sch:Memoirs}.
Let $B=B^{r_k,s_k}\ot B^{r_{k-1},s_{k-1}}\ot\cdots\ot B^{r_1,s_1}$. For
$b=b_k\ot b_{k-1}\ot\cdots\ot b_1\in B$ rotate each rectangular tableau $b_i$ by
$90^{\circ}$ clockwise to obtain $\tilde{b}_i$. Suppose the letter $a$ occurs 
in cell $c$ of $\tilde{b}_i$. Then replace letter $a$ in cell $c$ by $\tilde{a}$ 
where $\tilde{a}$ is chosen such that the letter $a$ in cell $c$ is the 
$\tilde{a}$-th letter $a$ in $row(b)$ reading from right to left. Finally,
turn each tableau up side down and define it to be $\tr(b)$. If $b\in\Path(B)$,
then $\tr(b)\in\Path(\tr(B))$ where 
$\tr(B)=B^{s_k,r_k}\ot B^{s_{k-1},r_{k-1}}\ot\cdots\ot B^{s_1,r_1}$.
The map $\tr$ satisfies $\tr^2=\id$.

\begin{example}
For a path $b$ in Example \ref{ex:path} we have 
\[
\tr(b)=\Yvcentermath1
\young(13,35,46)\ot\young(22)\ot\young(1,2,3,4)
\ot\young(1)\ot\young(1)\,.
\smallskip
\]
\end{example}

Finally, we define the map $S_m$ on $\Path(B)$.
For an element $b=b_k\ot b_{k-1}\ot\cdots\ot b_1$ of $B$ we define 
$S_m(b)=S_m(b_k)\ot S_m(b_{k-1})\ot\cdots\ot S_m(b_1)$, where $S_m$ on
a single KR crystal was defined in \S\ref{subsec:KRcrystal}.
If $b\in\Path(B)$, then we have $S_m(b)\in\Path(S_m(B))$, 
where $S_m(B)$ is obtained by
replacing each single KR crystal $B^{r,s}$ in $B$ with $B^{r,ms}$.
It is easy to see that $S_m\circ S_{m'}=S_{mm'}$.

\begin{example}
For a path $b$ in Example \ref{ex:path} we have
\[
S_2(b)=\Yvcentermath1
\young(111122,223344)\ot\young(22,33)\ot\young(11111133)
\ot\young(22)\ot\young(11)\,.
\smallskip
\]
\end{example}

\subsection{Operations on rigged configurations}

We define the corresponding operations on rigged configurations. We first define
$\delta,\beta^{(s)},\gamma^{(m)}$. Say a row of a rigged configuration singular
if its rigging is equal to the vacancy number $p^{(a)}_i$.

\begin{enumerate}
\item Suppose $L^{(1)}$ contains a row of length 1.
Set $\ell^{(0)}=1$ and repeat the following process for $a=1,2,\ldots,n$ 
or until stopped. Find the smallest integer $i\ge\ell^{(a-1)}$ such that 
there exists a singular row of length $i$ in $(\nu,J)^{(a)}$. 
If no such $i$ exists, set $\mathrm{rk}(\nu,J)=a$ and stop. Otherwise set $\ell^{(a)}=i$
and continue the process with $a+1$. Set all undefined $\ell^{(a)}$ to $\infty$.
As for the new multiplicity array $\tilde{L}$, $\tilde{L}^{(1)}$ is given by removing 
a row of length 1 from $L^{(1)}$ and other $L^{(a)}$ remain the same.
$\delta(\nu,J)$ is obtained by removing 
a box from the selected rows and making the new rows singular again.

\item Suppose $L^{(r)}$ contains a row of length $s$. 
$\tilde{L}$ is given by removing a row of length $s$ from $L^{(r)}$ and
adding a row of length $s$ to both $L^{(1)}$ and $L^{(r-1)}$ (two rows of length
$s$ to $L^{(1)}$ if $r=2$).
$\beta^{(s)}(\nu,J)$ is obtained by adding singular rows of length $s$ to
$(\nu,J)^{(a)}$ for $1\le a<r$.

\item Suppose $L^{(r)}$ contains a row of length $s$. 
$\tilde{L}$ is given by removing a row of length $s$ from $L^{(r)}$ and
adding a row of length $m$ and a row of length $s-m$.
We set $\gamma^{(m)}(\nu,J)=(\nu,J)$.
\end{enumerate}

These maps send a rigged configuration in $\RC(L)$ to another one in 
$\RC(\tilde{L})$. In \cite{Sch:Memoirs} operations $\beta^{(1)},\gamma^{(1)}$ 
were defined in the notations $j,i$. We set $\beta=\beta^{(1)},\gamma=\gamma^{(1)}$.

\begin{example}
Let $(\nu,J)$ be as in Example \ref{ex:RC}. Then $(\nu',J')=\gamma(\nu,J),
(\nu'',J'')=\beta(\nu',J'),(\nu''',J''')=\delta(\nu'',J'')$ are given successively as
follows.
%1
\vspace{7mm}
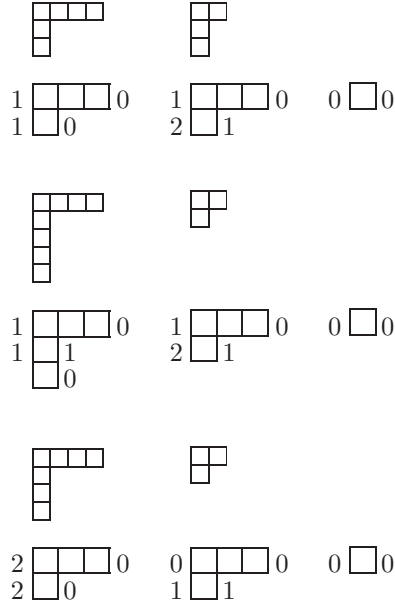
\begin{center}
\unitlength 10pt
\begin{picture}(20,6)
\Yboxdim{7pt}
\put(1,4){\yng(4,1,1)}
\put(7,4){\yng(2,1,1)}
\Yboxdim{10pt}
\put(1,1){\yng(3,1)}
\put(0.2,1){1}
\put(0.2,2){1}
\put(2.2,1){0}
\put(4.2,2){0}
\put(7,1){\yng(3,1)}
\put(6.2,1){2}
\put(6.2,2){1}
\put(8.2,1){1}
\put(10.2,2){0}
\put(13,2){\yng(1)}
\put(12.2,2){0}
\put(14.2,2){0}
\end{picture}
\end{center}
%2
\vspace{9mm}
\begin{center}
\unitlength 10pt
\begin{picture}(20,6)
\Yboxdim{7pt}
\put(1,5){\yng(4,1,1,1,1)}
\put(7,7.1){\yng(2,1)}
\Yboxdim{10pt}
\put(1,1){\yng(3,1,1)}
\put(0.2,2){1}
\put(0.2,3){1}
\put(2.2,1){0}
\put(2.2,2){1}
\put(4.2,3){0}
\put(7,2){\yng(3,1)}
\put(6.2,2){2}
\put(6.2,3){1}
\put(8.2,2){1}
\put(10.2,3){0}
\put(13,3){\yng(1)}
\put(12.2,3){0}
\put(14.2,3){0}
\end{picture}
\end{center}
%3
\vspace{7mm}
\begin{center}
\unitlength 10pt
\begin{picture}(20,6)
\Yboxdim{7pt}
\put(1,5){\yng(4,1,1,1)}
\put(7,6.4){\yng(2,1)}
\Yboxdim{10pt}
\put(1,2){\yng(3,1)}
\put(0.2,2){2}
\put(0.2,3){2}
\put(2.2,2){0}
\put(4.2,3){0}
\put(7,2){\yng(3,1)}
\put(6.2,2){1}
\put(6.2,3){0}
\put(8.2,2){1}
\put(10.2,3){0}
\put(13,3){\yng(1)}
\put(12.2,3){0}
\put(14.2,3){0}
\end{picture}
\end{center}
\end{example}

Next we define the map $\tr$ on rigged configurations. 
The new multiplicity array 
$\tilde{L}=(\tilde{L}^{(a)}_i)$ is defined by $\tilde{L}^{(a)}_i=L^{(i)}_a$.
We assume $n$ is sufficiently large so that $\tilde{L}$ is well defined.
For a configuration $\nu=(m^{(a)}_i)$ define a matrix $N=(N_{ai})_{a\in I_0,i\in\Z_{>0}}$
by
\begin{equation} \label{N}
N_{ai}=\sum_{j\ge i}\left(m^{(a-1)}_j-m^{(a)}_j\right).
\end{equation}
Note that $\sum_{j\ge i}m^{(a)}_i$ is the depth of the $i$-th column of $\nu^{(a)}$.
$m^{(0)}_j$ is defined to be zero for any $j$. Now set $(\tilde{\nu},\tilde{J})=\tr(\nu,J)$.
Then $\tilde{N}=N(\tilde{\nu})$ is given by
\begin{equation} \label{Ntilde}
\tilde{N}_{ia}=-N_{ai}+\chi((a,i)\in\la)-\sum_{b,j}L^{(b)}_j\chi(a\le b\,\&\,i\le j).
\end{equation}
Here $\chi(\theta)=1$ if $\theta$ is true, $=0$ otherwise. $\la$ is the partition 
$(\la_1,\la_2,\ldots)$ such that the weight of $(\nu,J)$ is given by
$\sum_j\la_j(\varpi_j-\varpi_{j-1})$ ($\varpi_0=0$),
and $(a,i)\in\la$ signifies that the cell of the $a$-th row and the $i$-th column
belongs to the Young diagram of $\la$.

Note that the rigging $\tilde{J}$ can be viewed as a double sequence of 
Young diagrams $J=(J^{(a,i)})$ where $J^{(a,i)}$ is a Young diagram inside 
the rectangle of depth $m^{(a)}_i$ and width $p^{(a)}_i$. The Young diagram
$\tilde{J}^{(i,a)}$ is defined as the transpose of the complementary Young
diagram to $J^{(a,i)}$ in the rectangle of depth $m^{(a)}_i$ and width $p^{(a)}_i$.
This $\tr$ also satisfies $\tr^2=\id$.
The following lemma, which is clear from the above rule, is used later.

\begin{lemma} \label{lem:tr and ls}
Let $B^{r,s}$ be the leftmost tensor factor of $B$. Then we have
$\ls\circ\tr=\tr\circ\lb^{(s)}$ on $\Path(B)$.
\end{lemma}

\begin{example}
Let $(\nu,J)$ be as in Example \ref{ex:RC}. Then we have
\[
\la=\hspace{-4mm}
\mbox{
\unitlength 10pt
\begin{picture}(5,2)
\Yboxdim{7pt}
\put(1,-1.2){\yng(6,4,3,1)}
\end{picture}
},
\quad
N=\left(
\begin{array}{ccc}
-2&-1&-1\\
0&0&0\\
1&1&1\\
1&0&0
\end{array}
\right),
\quad
\tilde{N}=\left(
\begin{array}{cc}
-2&-1\\
0&0\\
0&0\\
0&1\\
1&0\\
1&0
\end{array}
\right).
\]
All entries of $N,\tilde{N}$ outside the given part are zero. Thus $\tr(\nu,J)$ 
is given as follows.
\begin{center}
\unitlength 10pt
\begin{picture}(28,7)
\Yboxdim{7pt}
\put(1,4){\yng(2,1,1)}
\put(13,5.4){\yng(2)}
\put(19,5.4){\yng(1)}
\Yboxdim{10pt}
\put(1,1){\yng(2,1)}
\put(0.2,1){1}
\put(0.2,2){1}
\put(2.2,1){1}
\put(3.2,2){0}
\put(7,1){\yng(2,1)}
\put(6.2,1){0}
\put(6.2,2){0}
\put(8.2,1){0}
\put(9.2,2){0}
\put(13,1){\yng(2,1)}
\put(12.2,1){1}
\put(12.2,2){1}
\put(14.2,1){1}
\put(15.2,2){1}
\put(19,1){\yng(1,1)}
\put(18.2,2){0}
\put(20.2,1){0}
\put(20.2,2){0}
\put(25,2){\yng(1)}
\put(24.2,2){0}
\put(26.2,2){0}
\end{picture}
\end{center}
\end{example}

Finally, operation $S_m$ on rigged configurations is defined by enlarging all rows 
of $L^{(a)}$ and $\nu^{(a)}$ $m$ times and multiplying $m$ to all riggings. Note
that $p^{(a)}_{mi}(S_m(\nu))=mp^{(a)}_i(\nu)$.
It is again easy to see that $S_m\circ S_{m'}=S_{mm'}$.

\begin{example}
Let $(\nu,J)$ be as in Example \ref{ex:RC}. $S_2(\nu,J)$ is given below.
\vspace{7mm}
\begin{center}
\unitlength 10pt
\begin{picture}(22,5)
\Yboxdim{7pt}
\put(1,4){\yng(8,2,2)}
\put(10,4.7){\yng(6,2)}
\Yboxdim{10pt}
\put(1,1){\yng(6,2)}
\put(0.2,1){2}
\put(0.2,2){2}
\put(3.2,1){0}
\put(7.2,2){0}
\put(10,1){\yng(6,2)}
\put(9.2,1){2}
\put(9.2,2){2}
\put(12.2,1){2}
\put(16.2,2){0}
\put(19,2){\yng(2)}
\put(18.2,2){0}
\put(21.2,2){0}
\end{picture}
\end{center}
\end{example}

\subsection{KSS bijection and the main theorem}

For a tensor product of KR crystals $B=B^{r_k,s_k}\ot B^{r_{k-1},s_{k-1}}\ot\cdots\ot
B^{r_1,s_1}$ we define the corresponding multiplicity array $L(B)=(L_i^{(a)})$ by
\begin{equation} \label{def L}
L_i^{(a)}=\sharp\{j\mid (r_j,s_j)=(a,i),1\le j\le k\}.
\end{equation}

The following theorem is proved essentially in \cite{KSS}. We adopt the formulation
in \cite{Sch:Memoirs}.

\begin{theorem} \label{th:KSS}
Let $B=B^{1,1}\ot B'$ in (1), $B=B^{r,1}\ot B'$ ($r\ge2$) in (2), and $B=B^{r,s}\ot B'$
($s\ge2$) in (3). There exists a unique bijection $\Phi$ from $\Path(B)$ to 
$\RC(L(B))$ that maps the empty path to the empty rigged configurations
and makes the following diagrams commutative.
\begin{align*}
&(1)
\diagram{\Path(B)}{\RC(L(B))}{\Path(\lh(B))}{\RC(L(\lh(B)))}{\Phi}{\lh}{\delta}{\Phi}\quad
(2)
\diagram{\Path(B)}{\RC(L(B))}{\Path(\lb(B))}{\RC(L(\lb(B)))}{\Phi}{\lb}{\beta}{\Phi}
\\
&(3)
\diagram{\Path(B)}{\RC(L(B))}{\Path(\ls(B))}{\RC(L(\ls(B)))}{\Phi}{\ls}{\gamma}{\Phi}
\end{align*}
Moreover, under $\Phi$, $\tr$ corresponds to $\tr$ and $R_i$ corresponds to
$\id$ for any $i=1,2,\ldots,k-1$ where $R_i$ stands for the combinatorial $R$-matrix
acting on the $(i+1)$-th and the $i$-th position $B_{i+1}\ot B_i$ of $B$.
\end{theorem}

Now we can state our theorem.

\begin{theorem} \label{th:main}
Under $\Phi$, $S_m$ corresponds to $S_m$.
\end{theorem}

We need to prepare lemmas to prove it.

\begin{lemma}
Suppose $(\nu,J)\in\RC(L)$ with $L^{(1)}_1>0$. We apply $\delta$ on $(\nu,J)$.
Let $\ell^{(1)}\le\ell^{(2)}\le\cdots\le\infty$ be the lengths of rows whose box is 
removed from $\nu^{(1)},\nu^{(2)},\ldots$ by $\delta$ and set 
$(\tilde{\nu},\tilde{J})=\delta(\nu,J)$. Then we have
\[
p^{(a)}_i(\tilde{\nu})=p^{(a)}_i(\nu)-\chi(\ell^{(a-1)}\le i)+2\chi(\ell^{(a)}\le i)-\chi(\ell^{(a+1)}\le i).
\]
\end{lemma}

\begin{proof}
Easy from \eqref{p and Q}.
\end{proof}

\begin{lemma} \label{lem:m boxes}
Suppose $(\nu,J)\in\RC(L)$ with $L^{(1)}_1>0$.
Let $(\nu',J')=S_m(\nu,J)$. Then we can apply 
$\delta\circ(\delta\circ\gamma)^{m-1}$ on $(\nu',J')$.
During the operation $m$ boxes are removed from the same row in each $\nu'^{(a)}$.
\end{lemma}

\begin{proof}
Let $\ell^{(1)},\ell^{(2)},\ldots$ be as in the previous lemma.
Since a singular row remains singular and a nonsingular row does nonsingular by $S_m$, 
by the first $\delta\circ\gamma$
a box is removed from a row of length $m\ell^{(a)}$ in each $\nu'^{(a)}$.

Now we apply next $\delta\circ\gamma$ on $(\tilde{\nu},\tilde{J})=(\delta\circ\gamma)
(\nu',J')$. We show that during this process a box is removed from a singular row in
$\tilde{\nu}^{(a)}$ of length $m\ell^{(a)}-1$ proceeding with $a=1,2,\ldots$. Suppose
$\ell^{(a)}<\infty$. From the previous lemma, we have
\[
p^{(a)}_i(\tilde{\nu})-p^{(a)}_i(\nu')=\left\{
\begin{array}{ll}
0\quad&\text{if }i<m\ell^{(a-1)}\\
-1&\text{if }m\ell^{(a-1)}\le i<m\ell^{(a)}.
\end{array}
\right.
\]
Here $\ell^{(0)}$ should be understood as $1/m$. We look for a singular row
in $\tilde{\nu}^{(a)}$ of length not less than $m\ell^{(a-1)}-1$. 
Suppose $\ell^{(a-1)}<\ell^{(a)}$. In the interval $i<m\ell^{(a-1)}$ we only
need to consider the case of $i=m\ell^{(a-1)}-1$. However, there is no row
in $\tilde{\nu}^{(a)}$ of this length. In the interval $m\ell^{(a-1)}\le i<m\ell^{(a)}$
we have $p^{(a)}_i(\tilde{\nu})=p^{(a)}_i(\nu')-1$. However, all the riggings in 
$\tilde{\nu}^{(a)}$ are multiple of $m$ and hence remain nonsingular, 
except the one in the row of length 
$m\ell^{(a)}-1$ from which a box is removed in the previous $\delta$.
Since this row is singular, we remove a box from it. When $\ell^{(a-1)}=\ell^{(a)}$,
from the same reason a box is removed from a unique row of length $m\ell^{(a)}-1$.
Suppose now $\ell^{(a)}=\infty$. Then all rows of length not less than 
$m\ell^{(a-1)}-1$ are nonsingular.

We continue the application of $\delta(\circ\gamma)$ but the fact that a box is
removed from the same row in each $\nu'^{(a)}$ remains true.
\end{proof}

\begin{proposition} \label{prop:lb(s)}
Suppose the leftmost factor of $B$ is $B^{r,s}$. 
Then the following diagram commutes.
\[
\diagram{\Path(B)}{\RC(L(B))}{\Path(\lb^{(s)}(B))}
{\RC(L((\lb^{(s)}(B)))}{\Phi}{\lb^{(s)}}{\beta^{(s)}}{\Phi}
\]
\end{proposition}

\begin{proof}
Let $\beta^{(s)}(\nu,J)=(\nu',J')$, and $\tr(\nu,J)=(\tilde{\nu},\tilde{J}),
\tr(\nu',J')=(\tilde{\nu}',\tilde{J}')$. In view of Lemma \ref{lem:tr and ls},
Theorem \ref{th:KSS}(3) and the description of $\gamma$,
it is enough to show $\tilde{\nu}=\tilde{\nu}',\tilde{J}=\tilde{J}'$.
We first prove $\tilde{\nu}=\tilde{\nu}'$.

Let $m^{(a)}_j$ be defined as \eqref{nu} for $\nu$ and $m'^{(a)}_j$ for $\nu'$.
Then we have
\[
m'^{(a)}_j=m^{(a)}_j+\chi(a<r)\delta_{js}.
\]
Using a similar notation for $N_{ai}$ in \eqref{N} to $\nu$ we get
\[
N'_{ai}=N_{ai}+(\delta_{ar}-\delta_{a1})\chi(i\le s).
\]
Then by \eqref{Ntilde} one calculates
\begin{align*}
\tilde{N}'_{ia}&=-N'_{ai}+\chi((a,i)\in\la)-\sum_{b,j}L'^{(b)}_j\chi(a\le b\,\&\,i\le j)\\
&=-N_{ai}-(\delta_{ar}-\delta_{a1})\chi(i\le s)+\chi((a,i)\in\la)\\
&\qquad-\sum_{b,j}(L^{(b)}_j+(\delta_{b,r-1}+\delta_{b1}-\delta_{br})\delta_{js})
\chi(a\le b\,\&\,i\le j)\\
&=\tilde{N}_{ai},
\end{align*}
and hence we obtain $\tilde{\nu}=\tilde{\nu}'$.

To prove $\tilde{J}=\tilde{J}'$ note that $J^{(a,i)}=J'^{(a,i)}$ except when
$a<r$ and $i=s$, in which case $J'^{(a,i)}$ has an extra singular row to $J^{(a,i)}$.
Recalling $p^{(a)}_i=p'^{(a)}_i$ we obtain $\tilde{J}=\tilde{J}'$.
\end{proof}

\begin{proposition} \label{prop:R}
Let $b\in B^{r,s}$ and $1\le m<s$. Then we have
\[
R(\ls^{(m)}(b))=\ls^{(s-m)}(b).
\]
\end{proposition}

\begin{proof}
It is clear from the combinatorial description of $R$ in \S\ref{subsec:path}.
\end{proof}

\begin{proposition} \label{prop:ls(m)}
Suppose the leftmost factor of $B$ is $B^{r,s}$ with $s\ge2$. For any $m$
$(1\le m<s)$ the following diagram commutes.
\[
\diagram{\Path(B)}{\RC(L(B))}{\Path(\ls^{(m)}(B))}
{\RC(L((\ls^{(m)}(B)))}{\Phi}{\ls^{(m)}}{\gamma^{(m)}}{\Phi}
\]
\end{proposition}

\begin{proof}
We prove by induction on $m$. The $m=1$ case is nothing but Theorem \ref{th:KSS}
(3). Consider the following sequence of maps.
\begin{align*}
&B^{r,s}\ot B'\overset{\ls^{(m-1)}}{\longrightarrow}B^{r,m-1}\ot B^{r,s-m+1}\ot B'
\overset{R\ot1}{\longrightarrow}B^{r,s-m+1}\ot B^{r,m-1}\ot B' \\
&\overset{\ls}{\longrightarrow}B^{r,1}\ot B^{r,s-m}\ot B^{r,m-1}\ot B'
\overset{1\ot R\ot1}{\longrightarrow}B^{r,1}\ot B^{r,m-1}\ot B^{r,s-m}\ot B' \\
&\overset{\ls^{-1}}{\longrightarrow}B^{r,m}\ot B^{r,s-m}\ot B'
\end{align*}
At each step the maps cut, move the cutting line, or concatenate the tableau
belonging to $B^{r,s}$, by the definition of $\ls^{(k)}$ and Proposition \ref{prop:R}.
In particular, the last map $\ls^{-1}$ is well defined. It is clear that the composition
of these maps coincides with $\ls^{(m)}$. Since all maps in the sequence correspond
to the identity in the rigged configuration side by the induction hypothesis and 
Theorem \ref{th:KSS}, the proof is finished.
\end{proof}

\noindent{\it Proof of Theorem \ref{th:main}.}
First recall Lemma 5.3 given in \cite{KSS}. Consider the following diagram.
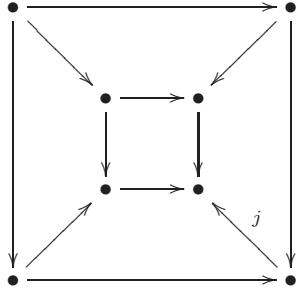
\[
\xymatrix{
\bullet \ar[rrr] \ar[dr] \ar[ddd] & & & \bullet  \ar[dl] \ar[ddd] \\
& \bullet  \ar[r] \ar[d] & \bullet  \ar[d] & \\
& \bullet  \ar[r] & \bullet  & \\
\bullet  \ar[ur] \ar[rrr] & & & \bullet  \ar[ul]_j
}
\]
Viewing this diagram as a cube with front face given by the large square, suppose
the square diagrams given by all faces of the cube except the front commute. 
Assume also that the map $j$ is injective. Then the front face should also commute.

For $B=B^{r_k,s_k}\ot B^{r_{k-1},s_{k-1}}\ot\cdots\ot B^{r_1,s_1}$ we introduce 
$(\sum_{j=1}^kr_js_j,\sum_{j=1}^k(r_j-1)s_j,\sum_{j=1}^k(s_j-1))$. Note that the operation
for $B$ in (i) ($i=1,2,3$) of Theorem \ref{th:KSS} decreases its $i$-th component
by $1$. We prove by induction on the lexicographic order of this index.

First suppose that the leftmost factor of $B$ is $B^{1,1}$ and 
consider the following diagram.
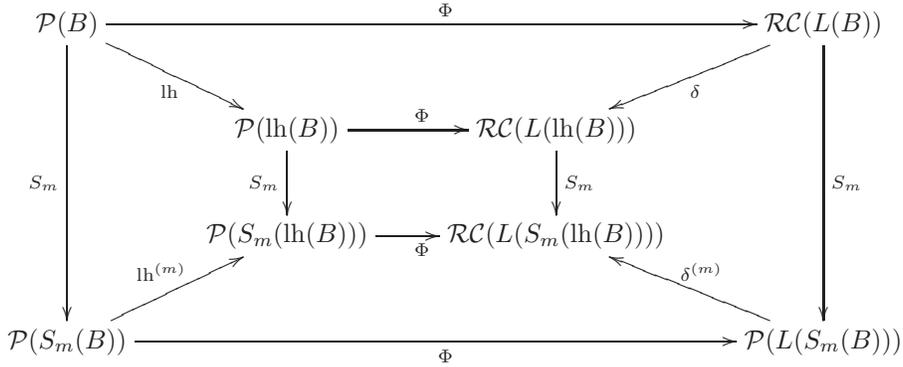
\[
\xymatrix{
\Path(B) \ar[rrr]^\Phi \ar[dr]_\lh \ar[ddd]_{S_m} & & & 
\RC(L(B)) \ar[dl]^{\delta} \ar[ddd]^{S_m} \\
& \Path(\lh(B)) \ar[r]^\Phi \ar[d]_{S_m} & \RC(L(\lh(B))) \ar[d]^{S_m} & \\
& \Path(S_m(\lh(B))) \ar[r]_\Phi & \RC(L(S_m(\lh(B)))) & \\
\Path(S_m(B)) \ar[ur]^{\lh^{(m)}} \ar[rrr]_\Phi & & & \Path(L(S_m(B))) \ar[ul]_{\delta^{(m)}}
}
\]
Here we have set $\lh^{(m)}=\lh\circ(\lh\circ\ls)^{m-1},
\delta^{(m)}=\delta\circ(\delta\circ\gamma)^{m-1}$.
We wish to show the front face commutes. By the above lemma and the injectivity
of $\delta^{(m)}$, it is enough to show all the other faces commute. The back face
is assumed to commute by induction. The commutativity of the left face is clear,
while the right face is due to Lemma \ref{lem:m boxes}. The top and bottom faces
commute by Theorem \ref{th:KSS}.

Next suppose that the leftmost factor of $B$ is $B^{r,1}$($r\ge2$).
Consider
\[
\xymatrix{
\Path(B) \ar[rrr]^\Phi \ar[dr]_\lb \ar[ddd]_{S_m} & & & 
\RC(L(B)) \ar[dl]^{\beta} \ar[ddd]^{S_m} \\
& \Path(\lb(B)) \ar[r]^\Phi \ar[d]_{S_m} & \RC(L(\lb(B))) \ar[d]^{S_m} & \\
& \Path(S_m(\lb(B))) \ar[r]_\Phi & \RC(L(S_m(\lb(B)))) & \\
\Path(S_m(B)) \ar[ur]^{\lb^{(m)}} \ar[rrr]_\Phi & & & \Path(L(S_m(B))) \ar[ul]_{\beta^{(m)}}
}
\]
Again, we show all the other faces other than the front one commute. The back face
commutes by induction. The commutativity of the left and right faces is clear.
The top face commutes by Theorem \ref{th:KSS}, while the bottom face is due to 
Proposition \ref{prop:lb(s)}.

Finally suppose that the leftmost factor of $B$ is $B^{r,s}$($s\ge2$).
Consider
\[
\xymatrix{
\Path(B) \ar[rrr]^\Phi \ar[dr]_\ls \ar[ddd]_{S_m} & & & 
\RC(L(B)) \ar[dl]^{\gamma} \ar[ddd]^{S_m} \\
& \Path(\ls(B)) \ar[r]^\Phi \ar[d]_{S_m} & \RC(L(\ls(B))) \ar[d]^{S_m} & \\
& \Path(S_m(\ls(B))) \ar[r]_\Phi & \RC(L(S_m(\ls(B)))) & \\
\Path(S_m(B)) \ar[ur]^{\ls^{(m)}} \ar[rrr]_\Phi & & & \Path(L(S_m(B))) \ar[ul]_{\gamma^{(m)}}
}
\]
In this case we use Proposition \ref{prop:ls(m)} for the commutativity of the 
bottom face.
\qed

\section{Application}

Theorem \ref{th:main} motivates us to extend the notions of $\Path(B)$ and $\RC(L)$. 
In the sequel we assume $R$ to be one of $\Z,\Q$ or $\R$. We set 
$R_{>0}=\{z\in R\mid z>0\},R_{\ge0}=\{z\in R\mid z\ge0\}$. 
First we recall an alternative description of $B^{r,s}$.
For a tableau in $B^{r,s}$ let $x_{i,j}$ ($1\le i\le r,i\le j\le n-r+i+1$) be the number
of boxes with letter $j$ in the $i$-th row. Then they must satisfy
\begin{align}
&\sum_{k=i}^{n-r+i+1}x_{i,k}=s\quad(i=1,2,\ldots,r), \label{sum of x}\\
&\sum_{k=i}^jx_{i,k}\ge\sum_{k=i+1}^{j+1}x_{i+1,k}\quad
(i=1,\ldots,r-1;j=i,\ldots,n-r+i). \label{x ineq}
\end{align}
Now for $r\in I_0,s\in R_{>0}$ define 
\[
\calB^{r,s}=\{(x_{i,j})\mid x_{i,j}\in R_{\ge0}\text{ and satisfies \eqref{sum of x} and
\eqref{x ineq}}\}.
\]
We can also consider a formal tensor product
\[
\calB=\calB^{r_k,s_k}\ot\cdots\ot \calB^{r_1,s_1}.
\]
Recall that $\Path(B)$ was defined as the set of elements $b$ satisfying
$e_ib=0$ for $i\in I_0$ in \S2.2. Note that $e_ib=0$ is equivalent to $\veps_i(b)=0$.
Hence our new notion $\Path_R(\calB)$ should be defined as 
the set of elements $b\in\calB$ satisfying 
$\veps_i(b)=0$ for $i\in I_0$ where $\veps_i(b)$ for $b\in \calB^{r,s}$ is given 
in \cite[\S5.2]{OSS} and $\veps_i(b)$ for a multiple tensor product $b\in \calB$
is calculated by extending \eqref{eps tensor} via coassociativity.
For $m\in R_{>0}$ the similarity map $S_m$ on $\Path_R(\calB)$ is defined as follows.
For $b=b_k\ot\cdots\ot b_1,b_l=(x_{i,j}^{(l)})\in \calB^{r_l,s_l}$, $S_m(b)=b'_k\ot
\cdots\ot b'_1$ is given by $b'_l=(mx_{i,j}^{(l)})$ for $1\le l\le k$. Then $S_m$ turns
out a map from $\Path_R(\calB)$ to $\Path_R(\calB')$ where $\calB'=
\calB^{r_k,ms_k}\ot\cdots\ot\calB^{r_1,ms_1}$. It is easy to see that 
$S_m\circ S_{m'}=S_{mm'}$.

The set of rigged configurations $\RC(L)$ defined in \S2.3 can also be extended by 
allowing the indices $i$ to take values in $R_{>0}$. Consider $L=(L_i^{(a)})_{a\in I_0,i\in R_{>0}}$ such that
$L_i^{(a)}\in\Z_{\ge0}$ and $L_i^{(a)}>0$ only for finitely many $(a,i)$. Let $\nu=(m_i^{(a)})_{a\in I_0,i\in R_{>0}}$
be similar such data. $\nu$ is said to be an admissible configuration if
it satisfies
\begin{equation*}
  p_i^{(a)} \ge 0\qquad\text{for any $a\in I_0$ and
  $i\in R_{>0}$,}
\end{equation*}
where
\begin{equation*}
p_i^{(a)} = \sum_{j\in R_{>0}} \left( L_j^{(a)} \min(i,j) -
\sum_{b\in I_0}C_{ab}\min(i,j)m_j^{(b)}\right).
\end{equation*}
We identify an admissible configuration $\nu$ with a sequence $(\nu^{(a)})_{a\in I_0}$ such that each 
$\nu^{(a)}$ is a weakly decreasing sequence of elements of $R_{>0}$ in which every $i\in R_{>0}$ appears
$m^{(a)}_i$ times. $L^{(a)}=(L_i^{(a)})_{i\in R_{>0}}$ and $\nu^{(a)}$ can be viewed as 
a generalized Young diagram in which each row can have 
length in the set $R_{>0}$. A rigging $J$ on $\nu$ is to associate, with each part of such generalized Young diagram
$\nu^{(a)}$ of the same width $i$, a sequence $(J^{(a,i)}_1\ge J^{(a,i)}_2\ge\ldots\ge J^{(a,i)}_{m^{(a)}_i})$
of length $m^{(a)}_i$ such that $J^{(a,i)}_j\in R_{\ge0}$ for $j=1,\ldots,m_i^{(a)}$ and $p^{(a)}_i\ge J^{(a,i)}_1$. 
We call such pair $(\nu,J)$ a rigged configuration and denote the set of rigged 
configurations by $\RC_R(L)$. For $m\in R_{>0}$ the similarity map $S_m$ on 
$\RC_R(L)$ is defined by enlarging all rows of $L^{(a)}$ and $\nu^{(a)}$ $m$ times
and lengthening all rows $m$ times and multiplying $m$ to all riggings. 
It is again easy to see that $S_m\circ S_{m'}=S_{mm'}$.

For $\calB=\calB^{r_k,s_k}\ot\cdots\ot \calB^{r_1,s_1}$ we define 
$L(\calB)=(L^{(a)}_i)_{a\in I_0,i\in R_{>0}}$ by \eqref{def L}. From Theorem \ref{th:KSS}
there is a bijection $\Phi_\Z=\Phi$ from $\Path_\Z(\calB)$ to $\RC_\Z(L(\calB))$.
We now define a map $\Phi_\Q$ from $\Path_\Q(\calB)$ to $\RC_\Q(L(\calB))$.
Let $b=b_k\ot\cdots\ot b_1,b_l=(x^{(l)}_{i,j})\in\calB^{r_l,s_l}$ be an element of
$\Path_\Q(\calB)$. Take the minimal positive integer $m_0$ such that 
$m_0x_{ij}^{(l)}\in\Z$ for all $i,j,l$. We define $\Phi_\Q(b)=
(S_{1/m_0}\circ\Phi\circ S_{m_0})(b)$.
\begin{proposition} \label{prop:Q}
This $\Phi_\Q$ is a bijection from $\Path_\Q(\calB)$ to $\RC_\Q(L(\calB))$.
\end{proposition}
\begin{proof}
For well-definedness we need to show that if $m$ is a multiple of $m_0$, the image 
of the above map $\Phi_\Q$ stays the same with $m_0$ replaced with $m$. But it is 
true since all the small diagrams 
below are commutative. (The middle one is from Theorem \ref{th:main}.)
\[\xymatrix{
\Path_\Q(\calB) \ar@{->}[r]^{S_{m_0}} \ar@{->}[rd]_{S_m} & 
\Path_\Z(\calB') \ar@{->}[r]^{\Phi_\Z} \ar@{->}[d]^{S_{m/m_0}} & 
\RC_\Z(L(\calB')) \ar@{->}[d]^{S_{m/m_0}} \ar@{->}[rd]^{S_{1/m_0}} \\
& \Path_\Z(\calB'') \ar@{->}[r]_{\Phi_\Z} & \RC_\Z(L(\calB'')) \ar@{->}[r]_{S_{1/m}} & 
\RC_\Q(L(\calB))
}\]
where $\calB'=\calB^{r_k,m_0s_k}\ot\cdots\ot\calB^{r_1,m_0s_1},
\calB''=\calB^{r_k,ms_k}\ot\cdots\ot\calB^{r_1,ms_1}$.
\end{proof}

Similarly, we can try to define a map $\Phi_\R$ from $\Path_\R(\calB)$ to 
$\RC_\R(L(\calB))$ by using the similarity map $S_m$ for $m\in\R$ and considering
a sequence of elements in $\Path_\Q(\calB')$ convergent to the element in
$\Path_\R(\calB)$. (The second index of each single KR crystal in $\calB'$ is
slightly shifted from $\calB$.) We conjecture that this map is well-defined, but its 
proof seems nontrivial. We also conjecture that $\Phi_\R$ is a homeomorphism.

\begin{remark}
An evidence of this conjecture is given in \cite{KSY}, where a piecewise linear formula
of the inverse map of $\Phi_\R$ is obtained when $r_j=1$ for any $j$.
\end{remark}

\begin{remark}
A similar map to our $\Phi_\R$ in the case of $A_1$ has been constructed in \cite{T}
to linearize a certain integrable system called the tropical periodic Toda lattice. It 
would be interesting to establish an explicit connection between them.
\end{remark}

\subsection*{Acknowledgements}
The author thanks Atsuo Kuniba, Anne Schilling, Taichiro Takagi and Yasuhiko Yamada
for useful discussion or comments. He is partially supported by the Grants-in-Aid 
for Scientific Research No. 23340007 from JSPS.

\end{document}